\documentclass{amsart}
\usepackage[utf8]{inputenc}
\usepackage{biblatex}
\usepackage{amssymb}
\usepackage{amsmath}
\usepackage{amsthm}
\usepackage{multirow}
\usepackage{longtable}
\usepackage{gensymb}
\usepackage{pst-node}
\usepackage{auto-pst-pdf}
\usepackage{tikz-cd}
\usepackage{tikz,tkz-euclide}
\usepackage{siunitx}
\usepackage{pgfplots}
\usepackage{float} 
\usepackage{xcolor}
\usepackage{textcomp}

\pgfplotsset{compat=1.17}

\theoremstyle{plain}
\newtheorem{prop}{Proposition}[section]

\newtheorem{thm}{Theorem}[section]

\theoremstyle{definition}
\newtheorem{eg}{Example}[section]
\newtheorem{rem}{Remark}[section]
\newtheorem{defn}{Definition}[section]

\newcommand{\F}{\mathbb F}

\title{The Mathieu group $M_{23}$ as additive functions on the finite field of size ${2^{11}}$}
\author[Y.B.,B.H.,R.H.,R.L., S.L., F.L., M.S., N.W,J.Y.,L.Z.,N.Z.]{ Yiming Bing*, Bright Hu, Ronni Hu, Rhianna Li, Stefan Lu, Finn McDonald, Michael Sun, Nicholas Wolfe, Joshua Yao, Leon Zhou, Nathan Zhou}

\begin{document}


                 
       


\begin{abstract}
   We explicitly extend the standard permutation action of the Mathieu group $M_{23}$ on a 23 element set $C=C_{23}$ contained in a finite field of $2^{11}$ elements $\mathbb{F}_{2^{11}}$ to additive functions on this finite field. That is we represent $M_{23}$ as functions $\varphi:\mathbb{F}_{2^{11}}\to \mathbb{F}_{2^{11}}$ such that $\varphi(x+y)=\varphi(x)+\varphi(y)$ and $\varphi|_{C}$ is the standard permutation action. We give explicit $11\times 11$ matrices for the pair of standard generators of order $23$ and order $5$, as well as many tables to help facilitate future calculations.
   
\end{abstract}

\maketitle
\section*{Introduction}
The Mathieu group $M_{23}$ was first introduced in 1873 in \cite{M1873}. (We translated this paper to English \cite{trans}). Since then, there has been much interest surrounding this group; for example, it is one of the sporadic simple finite groups in the classification program of finite simple groups and the only sporadic group for which the Inverse Galois Problem remains open. 

In 1957 \cite{P1957} proved the existence of an irreducible faithful representation of $M_{23}$ as $11\times 11$ matrices over $\mathbb{Z}/2\mathbb{Z}$ using the characterisation of $M_{23}$ as the automorphism group of a Steiner system, where certain key lemmas were left to computer calculations and omitted our of necessity. There were also many mentions of such a definition which restricts to the standard permutation action that can be found online such as Wikipedia or Mathoverflow but we were not able to find an explicit construction of this elegant characterisation in the literature. 

That is, let $C:=C_{23}$ be the $23$ element subgroup of the multiplicative group of the finite field of $2^{11}$ elements $\F_{2^{11}}$. We want to represent $M_{23}$ as invertible functions
$$\varphi: \F_{2^{11}}\to\F_{2^{11}}$$
such that $\varphi(C)\subseteq C$ and $\varphi(x+y)=\varphi(x)+\varphi(y)$ for all $x,y\in\F_{2^{11}}$ and the restriction
$$\varphi|_C:C\to C$$
is a corresponding permutation from $M_{23}$ acting transitively on a $23$ element set.

In this article we proceed to construct a representation satisfying these criteria very explicitly in the hopes that we and other young mathematicians can gain a better understanding and make an easier transition into the subject matter. We do not mention Steiner systems nor Golay codes to achieve this representation, which we believe will add much greater accessibility for young readers. All parts of the proof are directly verifiable and carried out by hand.

We also include extensive tables detailing many of the explicit relationships between the elements of the finite structures involved for completeness in the hopes that they also might be useful in future studies. 
We finally conclude with stating matrices for the generators of $M_{23}$.


\section{The Field $F_{2^{11}}$}
We construct $F_{2^{11}}$ using the irreducible polynomial $x^{11}+x^2+1$ mod 2
$$\F_{2^{11}}=F_{2^{11}}[x]/(x^{11}+x^2+1)F_{2^{11}}[x]\cong \{aX^2+bX+c\,|\,a,b,c\in F_{2^{11}}, X^{11}+X^2+1=0\}$$
and with $x \mapsto X$ we express everything as powers of $X$ in the first table. This polynomial appears on a list of irreducible polynomials in \cite{C1935} and our table also serves as an explicit proof that both the polynomial is irreducible and that the quotient by this polynomial is a field with $2^{11}$ elements. The polynomial was chosen because it has the fewest terms with the lowest power of $x$. It is clear that every degree $11$ polynomial will have the $x^{11}$ term and if irreducible will have constant term 1 and at least 3 terms ($x^{11}+1$ is reducible, for example substitute $x=1$). It suffices to check that $x^{11}+x+1$ is reducible.

While it is known the multiplicative group is cyclic (can be expressed as powers of a single element), our table below also proves it explicitly here using the $X$ as the generator. It suffices to check that $X^{23}\neq1$ and $X^{89}\neq1$. The initials track who participated in the calculations. Many of these entries are not needed in our proof but are included for completeness.

\small
\begin{center}
\begin{longtable}[c]{|c|c|c|c|}
\hline
Power& Expression& Binary String& Initials \\
\hline
$X^{11}$&$X^2+1$&$00000000101$& MS\\
$X^{12}$&$X^3+X$&$00000001010$& BH\\
$X^{13}$&$X^4+X^2$&$00000010100$& BH\\
$X^{14}$&$X^5+X^3$&$00000101000$& BH\\
$X^{15}$&$X^6+X^4$&$00001010000$& BH\\
$X^{16}$&$X^7+X^5$&$00010100000$& FM, BH\\
$X^{17}$&$X^8+X^6$&$00101000000$& BH\\
$X^{18}$&$X^9+X^7$&$01010000000$& BH\\
$X^{19}$&$X^{10}+X^8$&$10100000000$& BH\\
$X^{20}$&$X^9+X^2+1$&$01000000101$& MS, BH\\
$X^{21}$&$X^{10}+X^3+X$&$10000001010$& BH\\
$X^{22}$&$X^4+1$&$00000010001$& FM\\
$X^{23}$&$X^5+X$&$00000100010$& FM\\
$X^{24}$&$X^6+X^2$&$00001000100$& FM\\
$X^{25}$&$X^7+X^3$&$00010001000$& FM\\
$X^{26}$&$X^8+X^4$&$00100010000$& FM\\
$X^{27}$&$X^9+X^5$&$01000100000$& FM\\
$X^{28}$&$X^{10}+X^6$&$10001000000$& FM\\
$X^{29}$&$X^7+X^2+1$&$00010000101$& FM\\
$X^{30}$&$X^8+X^3+X$&$00100001010$& FM\\
$X^{31}$&$X^9+X^4+X^2$&$01000010100$& FM\\
$X^{32}$&$X^{10}+X^5+X^3$&$10000101000$&FM\\
$X^{33}$&$X^6+X^4+X^2+1$&$00001010101$& FM\\
$X^{34}$&$X^7+X^5+X^3+X$&$00010101010$& BH\\
$X^{35}$&$X^8+X^6+X^4+X^2$&$00101010100$& BH\\
$X^{36}$&$X^9+X^7+X^5+X^3$&$01010101000$& BH\\
$X^{37}$&$X^{10}+X^8+X^6+X^4$&$10101010000$& BH\\
$X^{38}$&$X^9+X^7+X^5+X^2+1$&$01010100101$& BH\\
$X^{39}$&$X^{10}+X^8+X^6+X^3+X$&$10101001010$& BH\\
$X^{40}$&$X^9+X^7+X^4+1$&$01010010001$& MS, BH\\
$X^{41}$&$X^{10}+X^8+X^5+X$&$10100100010$& BH\\
$X^{42}$&$X^9+X^6+1$&$01001000001$& BH\\
$X^{43}$&$X^{10}+X^7+X$&$10010000010$& BH\\
$X^{44}$&$X^8+1$&$00100000001$& BH\\
$X^{45}$&$X^9+X$&$01000000010$& BH\\
$X^{46}$&$X^{10}+X^2$&$10000000100$& BH\\
$X^{47}$&$X^3+X^2+1$&$00000001101$& BH\\
$X^{48}$&$X^4+X^3+X$&$00000011010$& BH\\
$X^{49}$&$X^5+X^4+X^2$&$00000110100$& BH\\
$X^{50}$&$X^6+X^5+X^3$&$00001101000$& BH\\
$X^{51}$&$X^7+X^6+X^4$&$00011010000$& BH\\
$X^{52}$&$X^8+X^7+X^5$&$00110100000$& BH\\
$X^{53}$&$X^9+X^8+X^6$&0$1101000000$& BH\\
$X^{54}$&$X^{10}+X^9+X^7$&$11010000000$& BH\\
$X^{55}$&$X^{10}+X^8+X^2+1$&$10100000101$& RK\\
$X^{56}$&$X^9+X^3+X^2+X+1$&$01000001111$& BH\\
$X^{57}$&$X^{10}+X^4+X^3+X^2+X$&$10000011110$& BH\\
$X^{58}$&$X^5+X^4+X^3+1$&$00000111001$& BH\\
$X^{59}$&$X^6+X^5+X^4+X$&$00001110010$& BH\\
$X^{60}$&$X^7+X^6+X^5+X^2$&$00011100100$& BH\\
$X^{61}$&$X^8+X^7+X^6+X^3$&$00111001000$& BH\\
$X^{62}$&$X^9+X^8+X^7+X^4$&$01110010000$& BH\\
$X^{63}$&$X^{10}+X^9+X^8+X^5$&$11100100000$& BH\\
$X^{64}$&$X^{10}+X^9+X^6+X^2+1$&$11001000101$& MS,BH\\
$X^{65}$&$X^{10}+X^7+X^3+X^2+X+1$&$10010001111$& BH\\
$X^{66}$&$X^8+X^4+X^3+X+1$&$00100011011$& FM\\
$X^{67}$&$X^9+X^5+X^4+X^2+X$&$01000110110$& BH\\
$X^{68}$&$X^{10}+X^6+X^5+X^3+X^2$&$10001101100$& BH\\
$X^{69}$&$X^7+X^6+X^4+X^3+X^2+1$&$00011011101$& BH\\
$X^{70}$&$X^8+X^7+X^5+X^4+X^3+X$&$00110111010$& BH\\ 
$X^{71}$&$X^9+X^8+X^6+X^5+X^4+X^2$&$01101110100$& BH\\
$X^{72}$&$X^{10}+X^9+X^7+X^6+X^5+X^3$&$11011101000$& BH\\
$X^{73}$&$X^{10}+X^8+X^7+X^6+X^4+X^2+1$&$10111010101$& BH\\
$X^{74}$&$X^9+X^8+X^7+X^5+X^3+X^2+X+1$&$01110101111$& BH\\
$X^{75}$&$X^{10}+X^9+X^8+X^6+X^4+X^3+X^2+X$&$11101011110$& BH\\
$X^{76}$&$X^{10}+X^9+X^7+X^5+X^4+X^3+1$&$11010111001$& BH\\
$X^{77}$&$X^{10}+X^8+X^6+X^5+X^4+X^2+X+1$&$10101110111$& FM\\
$X^{78}$&$X^9+X^7+X^6+X^5+X^3+X+1$&$01011101011$& BH\\
$X^{79}$&$X^{10}+X^8+X^7+X^6+X^4+X^2+X$&$10111010110$& BH\\
$X^{80}$&$X^9+X^8+X^7+X^5+X^3+1$&$01110101001$& BH\\
$X^{81}$&$X^{10}+X^9+X^8+X^6+X^4+X$&$11101010010$& BH\\
$X^{82}$&$X^{10}+X^9+X^7+X^5+1$&$11010100001$& BH\\
$X^{83}$&$X^{10}+X^8+X^6+X^2+X+1$&$10101000111$& BH\\
$X^{84}$&$X^9+X^7+X^3+X+1$&$01010001011$& BH\\
$X^{85}$&$X^{10}+X^8+X^4+X^2+X$&$10100010110$& BH\\
$X^{86}$&$X^9+X^5+X^3+1$&$01000101001$& BH\\
$X^{87}$&$X^{10}+X^6+X^4+X$&$10001010010$& BH\\
$X^{88}$&$X^7+X^5+1$&$00010100001$& FM\\
$X^{89}$&$X^8+X^6+X$&$00101000010$& BH\\
$X^{90}$&$X^9+X^7+X^2$&$01010000100$& BH\\
$X^{91}$&$X^{10}+X^8+X^3$&$10100001000$& BH\\
$X^{92}$&$X^9+X^4+X^2+1$&$01000010101$& BH\\
$X^{93}$&$X^{10}+X^5+X^3+X$&$10000101010$& BH\\
$X^{94}$&$X^6+X^4+1$&$00001010001$& BH\\
$X^{95}$&$X^7+X^5+X$&$00010100010$& BH\\
$X^{96}$&$X^8+X^6+X^2$&$00101000100$& RH, MS, BH\\
$X^{97}$&$X^9+X^7+X^3$&$01010001000$& BH\\
$X^{98}$&$X^{10}+X^8+X^4$&$10100010000$& BH\\
$X^{99}$&$X^9+X^5+X^2+1$&$01000100101$& BH\\
$X^{100}$&$X^{10}+X^6+X^3+X$&$10001001010$& BH\\
\hline
\end{longtable}
\end{center}

\subsection{The Powers of $X^{89}$}

\normalsize
Let $\alpha=X^{89}$, here are the elements of $C$.
\begin{center}
\begin{tabular}{|c|c|c|c|}
\hline
Power& Expression& Binary& Initials \\
\hline
$\alpha$&$X^8+X^6+X$& $00101000010$ & MS\\
$\alpha^2$&$X^7+X^5+X^3+X^2+X$&$00010101110$&MS\\
$\alpha^3$&$X^{10}+X^7+X^3+X^2$&$10010001100$&MS\\
$\alpha^4$&$X^{10}+X^6+X^5+X^4+X^3+X^2$&$10001111100$&MS\\
$\alpha^5$&$X^8+X^7+X^6+X^5+1$&$00111100001$& SL, BH, YB\\
$\alpha^6$&$X^9+X^6+X^5+X^4+X^3+X^2+1$&$01001111101$&ML, RK\\
$\alpha^7$&$X^{10}+X^9+X^8+X^7+X^4+X^2+X$&$11110010110$& RH, MS\\
$\alpha^8$&$X^{10}+X^9+X^8+X^6+X^4+X^3+X^2+X+1$&$11101011111$ & RH, MS\\
$\alpha^9$&$X^{10}+X^8+X^7+X^6+X^5+X^3+X^2+X$&$10111101110$& RL\\
$\alpha^{10}$&$X^{10}+X^7+X+1$&$10010000011$& JY, LZ\\
$\alpha^{11}$&$X^7+X^5+X^2+X+1$&$00010100111$& MS, RH\\
$\alpha^{12}$&$X^{10}+X^9+X^8+X^7+X^6+X^4+X^3+X+1$& $11111011011$ & RH,MS,JY,YB\\
$\alpha^{13}$&$X^8+X^7+X^5+X$&$00110100010$&  RL, MS, JY\\
$\alpha^{14}$&$X^8+X^4+X^3+1$&$00100011001$&MS,LZ,RL\\
$\alpha^{15}$&$X^{10}+X^8+X^7+X^6+X^5+X^4+X^2+1$&$11010011110$&MS,RL,JY\\
$\alpha^{16}$&$X^8+X^6+X^5+X^4+X^3+X$&$00101111010$& MS, BH, RL \\
$\alpha^{17}$&$X^{10}+X^9+X^7+X^6$&$11011000000$ & RH,MS,LZ,YB\\
$\alpha^{18}$&$X^{10}+X^9+X^7+X^6+X^4+X+1$&$11011010011$& YB, BH\\
$\alpha^{19}$&$X^{8}+X^5+X^4+X^3+X^2+X+1$&$00100111111$& NZ, MS\\
$\alpha^{20}$&$X^9+X^5+X^3$&$01000101000$& BH\\
$\alpha^{21}$&$X^{10}+X^9+X^8+X^6+X^4+X^2$&$11101010100$& NW\\
$\alpha^{22}$&$X^{10}+X^5+X^4+X^3+X^2+1$&$10000111101$& YB\\
$\alpha^{23}$&$1$&$00000000001$& MS,BH\\
$\alpha^{24}$&$X^8+X^6+X$&$00101000010$&YB\\
\hline
\end{tabular}
\end{center}
    
    Notes: 1,2,3,4,5,10,20 powers are used to calculate the 23rd power of alpha as one so they are very likely correct. Similarly with 2,5,16,21 (NW). Also checked 11,12,22. 17,18,19,20 all checked.\\

\subsection{Powers of X As Powers of $\alpha$}



%




Let $A=\{1,\alpha,\alpha^2,\dots,\alpha^{10}\}$, we see from the table below that this is a additively independent set. This is because we know the set $\chi=\{1,X,X^2,\dots,X^{10}\}$ is additively independent and we write each $x\in\chi$ as combinations of $a\in A$.

\begin{center}
\begin{tabular}{|c|c|c|}
\hline
Power& Expression& Initials \\
\hline
    $X^0$&$1$& BH, YB \\
    $X^1$&$\alpha^{10}+\alpha^5+\alpha^3+\alpha^2+\alpha$& MS, LS\\
    $X^2$&$\alpha^9+\alpha^7+\alpha^6+1$ & MS, JG, TC, LS, JT\\
    $X^3$&$\alpha^9+\alpha^7+\alpha^6+\alpha^5+\alpha^2+\alpha$ & MS,LZ,NZ\\
    $X^4$&$\alpha^8+\alpha^6+\alpha^5+\alpha^4+\alpha^2+1$& LZ, NW, MS\\
    $X^5$&$\alpha^9+\alpha^3+\alpha$ & RH, MS\\
    $X^6$&$\alpha^{10}+\alpha^8+\alpha^6+\alpha^5$& LZ, FM, TC\\
    $X^7$&$\alpha^{10}+\alpha^9+\alpha^2+\alpha+1$& MS\\
    $X^8$&$\alpha^{8}+\alpha^6+\alpha^3+\alpha^2$ & RH, MS, LZ, NW\\
    $X^9$&$\alpha^{10}+\alpha^9+\alpha^6+\alpha^5+\alpha^4+\alpha^3+1$ & YB\\
    $X^{10}$&$\alpha^{10}+\alpha^9+\alpha^5+\alpha^3$& MS, JG\\
    \hline
\end{tabular}
\end{center}

We list some examples of additive functions on $\F_{2^{11}}$ that preserve $C$.
\begin{eg}
$$x\mapsto \alpha x.$$
\end{eg}
\begin{eg}
$$\mathrm{Fr}(x)=x^2$$
\end{eg}

\subsection{Powers of $\alpha$ as the Sum of Lower Powers of $\alpha$} 

\begin{center}
\begin{tabular}{|c|c|c|}
\hline
    Power & Expression & Initial\\
\hline
    $\alpha^{11}$ & $\alpha^9+\alpha^7+\alpha^6+\alpha^5+\alpha+1 $ & MS, RL\\ 
    $\alpha^{12}$ & $\alpha^{10}+\alpha^8+\alpha^7+\alpha^6+\alpha^2+\alpha$ & MS\\ 
    $\alpha^{13}$ & $\alpha^8+\alpha^6+\alpha^5+\alpha^3+\alpha^2+\alpha+1$ & MS, RL, LZ\\
    $\alpha^{14}$ & $\alpha^9+\alpha^7+\alpha^6+\alpha^4+\alpha^3+\alpha^2+\alpha$ & MS, RH\\
    $\alpha^{15}$ & $\alpha^{10}+\alpha^8+\alpha^7+\alpha^5+\alpha^4+\alpha^3+\alpha^2$ & BH\\
    $\alpha^{16}$ & $\alpha^8+\alpha^7+\alpha^4+\alpha^3+\alpha+1$ & MS\\
    $\alpha^{17}$ & $\alpha^9+\alpha^8+\alpha^5+\alpha^4+\alpha^2+\alpha$ & MS \\
    $\alpha^{18}$ & $\alpha^{10}+\alpha^9+\alpha^6+\alpha^5+\alpha^3+\alpha^2$ & MS \\
    $\alpha^{19}$ & $\alpha^{10}+\alpha^9+\alpha^5+\alpha^4+\alpha^3+\alpha+1$ & MS, BH\\
    $\alpha^{20}$ &
    $\alpha^{10}+\alpha^9+\alpha^7+\alpha^4+\alpha^2+1$ & LZ, BH\\
    $\alpha^{21}$ & 
    $\alpha^{10}+\alpha^9+\alpha^8+\alpha^7+\alpha^6+\alpha^3+1$ & MS, FM\\
    $\alpha^{22}$ &
    $\alpha^{10}+\alpha^8+\alpha^6+\alpha^5+\alpha^4+1$ & LZ, BH\\
    $\alpha^{23}$ &
    $1$ & BH, MS, YB\\
    \hline
\end{tabular}
\end{center}
Here is an example which does not preserve $C$.
\begin{eg}
Example of function defined on 11 elements of C:
$f(1)=1$, $f(\alpha)=\alpha^2$, $f(\alpha^2)=\alpha$, $f(\alpha^i)=\alpha^i$ for $i\neq1,2$ $i<11$.
no adjacent transposition can be together with a full cycle as otherwise we get $S_{23}$.

Explicity,
$f(\alpha^{11})= \alpha^9+\alpha^7+\alpha^6+\alpha^5+\alpha^2+1\notin C$ 
\end{eg}

\section{Extending functions}

\subsection{Powers of $\beta$ as Powers of $\alpha$ and vice versa}
\begin{center}
\begin{tabular}{|c|c|c|c|}
\hline
    Power of $\beta$& Cor. Power of $\alpha$& Power of $\alpha$& Cor. Power of $\beta$\\\hline
$\beta$ & $\alpha^5$ &$\alpha$ & $\beta^{14}$\\
$\beta^2$ & $\alpha^{10}$ &$\alpha^2$ & $\beta^5$\\
$\beta^3$ & $\alpha^{15}$ &$\alpha^3$ & $\beta^{19}$ \\
$\beta^4$ & $\alpha^{20}$ &$\alpha^4$ & $\beta^{10}$\\
$\beta^5$ & $\alpha^2$ &$\alpha^5$ & $\beta$\\
$\beta^6$ & $\alpha^7$ &$\alpha^6$ & $\beta^{15}$ \\
$\beta^7$ & $\alpha^{12}$ &$\alpha^7$ & $\beta^6$ \\
$\beta^8$ & $\alpha^{17}$ & $\alpha^8$ & $\beta^{20}$\\
$\beta^9$ & $\alpha^{22}$ & $\alpha^9$ & $\beta^{11}$\\
$\beta^{10}$ & $\alpha^4$ & $\alpha^{10}$ & $\beta^2$\\
$\beta^{11}$ & $\alpha^9$ & $\alpha^{11}$ & $\beta^{16}$\\
$\beta^{12}$ & $\alpha^{14}$ & $\alpha^{12}$ & $\beta^7$ \\
$\beta^{13}$ & $\alpha^{19}$ & $\alpha^{13}$ & $\beta^{21}$ \\
$\beta^{14}$ & $\alpha$ & $\alpha^{14}$ & $\beta^{12}$\\
$\beta^{15}$ & $\alpha^6$ & $\alpha^{15}$ & $\beta^3$\\
$\beta^{16}$ & $\alpha^{11}$ & $\alpha^{16}$ & $\beta^{17}$\\
$\beta^{17}$ & $\alpha^{16}$ & $\alpha^{17}$ & $\beta^8$ \\
$\beta^{18}$ & $\alpha^{21}$ & $\alpha^{18}$ & $\beta^{22}$\\
$\beta^{19}$ & $\alpha^3$ & $\alpha^{19}$ & $\beta^{13}$\\
$\beta^{20}$ & $\alpha^8$ & $\alpha^{20}$ & $\beta^4$\\
$\beta^{21}$ & $\alpha^{13}$ & $\alpha^{21}$ & $\beta^{18}$\\
$\beta^{22}$ & $\alpha^{18}$ & $\alpha^{22}$ & $\beta^9$\\
$\beta^{23}$ & $\alpha^{23}$ & $\alpha^{23}$ & $\beta^{23}$\\ 
\hline
\end{tabular}
\end{center}

Checked by YB

Let $A=\{1,\alpha,\alpha^2,\dots,\alpha^{10}\}$ and $\chi=\{1,X,X^2,\dots,X^{10}\}$ be the additively independent sets from previous sections.
\begin{defn}
Let $f:C\to C$
by $f(\beta^j)=\beta^{j+1}$. Under the identification 
$$\beta^j\mapsto j$$
$f$ can be given by the cycle
$$(1,2,3,\dots,22,23).$$
\end{defn}
\begin{defn}
Let $g: C\to C$ be defined as
$$(3,17,10,7,9)(4,13,14,19,5)(8,18,11,12,23)(15,20,22,21,16)$$
under the identification
$$\beta^j\mapsto j.$$
\end{defn}

\begin{prop}[Mathieu \cite{M1873}]\label{gens}
The functions $f$ and $g$ generate $M_{23}$ inside the symmetric group $S_{23}$.
\end{prop}

\begin{prop}
The function $f$ extends uniquely to an additive function on $\F_{2^{11}}$. Moreover we have
$$f(x)=\beta x$$
for all $x\in\F_{2^{11}}$.
\end{prop}
\begin{proof}
$f$ is uniquely determined as an additive function by its values on the first 11 powers of alpha, while we note that multiplication by $\beta$ is clearly additive and agrees with $f$ on $C$. Hence $f$ extends uniquely to multiplication by $\beta$ and is additive.
\end{proof}

\begin{thm}The function $g$ extends uniquely to an additive function on $\F_{2^{11}}$.
\end{thm}
\begin{proof} It suffices to show that outputs of $\alpha^{11}$ to $\alpha^{22}$ agree with those determined by additivity on the additively independent set $1,\alpha,\alpha^2,\dots,\alpha^{10}$. The remaining outputs are then defined uniquely by additivity and the inputs in $A$.

$$\begin{aligned}
g(\alpha^{11})
    &=g(\alpha^9+\alpha^{7}+\alpha^{6}+\alpha^{5}+\alpha^1+1)\\
    &=g(\beta^{11}+\beta^{6}+\beta^{15}+\beta+\beta^{14}+\beta^{23})\\
    &=g(\beta^{11})+g(\beta^{6})+g(\beta^{15})+g(\beta)+g(\beta^{14})+g(\beta^{23})\\
    &=\beta^{12}+\beta^{6}+\beta^{20}+\beta+\beta^{19}+\beta^8\\
    &= \alpha^{14}+\alpha^{7}+\alpha^{8}+\alpha^{5}+\alpha^{3}+\alpha^{17}\\         
    &= (\alpha^{9}+\alpha^{7}+\alpha^{6}+\alpha^{4}+\alpha^{3}+\alpha^{2}+\alpha)+\alpha^{7}+\alpha^{8}+\alpha^{5}+\alpha^{3}+\alpha^{17}\\
    &= \alpha^{9}+\alpha^{8}+\alpha^{6}+\alpha^{5}+\alpha^{4}+\alpha^{2}+\alpha+\alpha^{17}\\
    &=\alpha^{9}+\alpha^{8}+\alpha^{6}+\alpha^{5}+\alpha^{4}+\alpha^{2}+\alpha+(\alpha^{9}+\alpha^{8}+\alpha^{5}+\alpha^{4}+\alpha^{2}+\alpha)\\
    &=\alpha^6,\qquad g(\beta^{16})=\beta^{15}, \qquad \text{YB, RL}.
\end{aligned}$$

$$\begin{aligned}
g(\alpha^{12})
    &=g(\alpha^{10}+\alpha^{8}+\alpha^{7}+\alpha^6+\alpha^2+\alpha)\\
    &=g(\beta^{2}+\beta^{20}+\beta^{6}+\beta^{15}+\beta^{5}+\beta^{14})\\
    &=g(\beta^{2})+g(\beta^{20})+g(\beta^{6})+g(\beta^{15})+g(\beta^{5})+g(\beta^{14})\\
    &=\beta^{2}+\beta^{22}+\beta^{6}+\beta^{20}+\beta^{4}+\beta^{19}\\
    &= \alpha^{10}+\alpha^{18}+\alpha^{7}+\alpha^{8}+\alpha^{20}+\alpha^{3}\\                                                 
    &= \alpha^{10}+(\alpha^{10}+\alpha^{9}+\alpha^{6}+\alpha^{5}+\alpha^{3}+\alpha^{2})+\alpha^{7}+\alpha^{8}+\alpha^{20}+\alpha^{3}\\ 
    &=\alpha^{9}+\alpha^{6}+\alpha^{5}+\alpha^{2}+\alpha^{7}+\alpha^{8}+\alpha^{20}\\
    &=\alpha^{9}+\alpha^{6}+\alpha^{5}+\alpha^{2}+\alpha^{7}+\alpha^{8}+(\alpha^{10}+\alpha^{9}+\alpha^{7}+\alpha^{4}+\alpha^{2}+1)\\
    &=\alpha^{6}+\alpha^{5}++\alpha^{8}+(\alpha^{10}+\alpha^{4}+1)\\
    &=\alpha^{10}+\alpha^{8}+\+\alpha^{6}+\alpha^{5}+\alpha^{4}+1\\
    &=\alpha^{22}, \qquad g(\beta^7)=\beta^9,\qquad \text{YB, RL} 
\end{aligned}$$                               
          
$$\begin{aligned}
g(\alpha^{13})
    &=g(\alpha^{8}+\alpha^{6}+\alpha^{5}+\alpha^{3}+\alpha^{2}+\alpha+1)\\
    &=g(\beta^{20}+\beta^{15}+\beta+\beta^{19}+\beta^{5}+\beta^{14}+\beta^{23})\\
    &=\beta^{22}+\beta^{20}+\beta+\beta^{5}+\beta^{4}+\beta^{19}+\beta^{8}\\
    &= \alpha^{18}+\alpha^{8}+\alpha^{5}+\alpha^2+\alpha^{20}+\alpha^{3}+\alpha^{17}\\
    &= (\alpha^{10}+\alpha^{9}+\alpha^{6}+\alpha^{5}+\alpha^{3}+\alpha^{2})+\alpha^{8}+\alpha^{5}+\alpha^{2}+\alpha^{3}+\alpha^{20}+\alpha^{17}\\
    &= \alpha^{10}+\alpha^{9}+\alpha^{8}+\alpha^{6}+(\alpha^{10}+\alpha^{9}+\alpha^{7}+\alpha^{4}+\alpha^{2}+1)+\alpha^{17}\\
    &= \alpha^{8}+\alpha^{7}+\alpha^{6}+\alpha^{4}+\alpha^{2}+1+(\alpha^{9}+\alpha^{8}+\alpha^{5}+\alpha^{4}+\alpha^{2}+\alpha)\\
    &= \alpha^{9}+\alpha^{7}+\alpha^{6}+\alpha^{5}+\alpha+1\\
    &= \alpha^{11},\qquad g(\beta^{21})=\beta^{16},\qquad\text{YB, RL, LZ}
\end{aligned}$$ 

$$\begin{aligned}
g(\alpha^{14})
    &=g(\alpha^{9}+\alpha^{7}+\alpha^{6}+\alpha^{4}+\alpha^{3}+\alpha^{2}+\alpha)\\
    &=g(\beta^{11}+\beta^{6}+\beta^{15}+\beta^{10}+\beta^{19}+\beta^{5}+\beta^{14})\\
    &=g(\beta^{11})+g(\beta^{6})+g(\beta^{15})+g(\beta^{10})+g(\beta^{19})+g(\beta^{5})+g(\beta^{14})\\
    &=\beta^{12}+\beta^{6}+\beta^{20}+\beta^{7}+\beta^{5}+\beta^{4}+\beta^{19}\\
    &=\alpha^{14}+\alpha^{7}+\alpha^{8}+\alpha^{12}+\alpha^{2}+\alpha^{20}+\alpha^{3}\\
    &=(\alpha^{9}+\alpha^{7}+\alpha^{6}+\alpha^{4}+\alpha^{3}+\alpha^{2}+\alpha)+\alpha^{7}+\alpha^{8}+\alpha^{2}+\alpha^{3}+\alpha^{12}+\alpha^{20}\\
    &=\alpha^{9}+\alpha^{8}+\alpha^{6}+\alpha^{4}+\alpha+(\alpha^{10}+\alpha^{8}+\alpha^{7}+\alpha^{6}+\alpha^{2}+\alpha)+\alpha^{20}\\
    &=\alpha^{10}+\alpha^{9}+\alpha^{7}+\alpha^{4}+\alpha^{2}+(\alpha^{10}+\alpha^{9}+\alpha^{7}+\alpha^{4}+\alpha^{2}+1)\\
    &= 1\\
    &=\alpha^{23},\qquad g(\beta^{12})=\beta^{23},\qquad\text{YB, RL, NZ}
\end{aligned}$$ 

$$\begin{aligned}
g(\alpha^{15})&=g(\alpha^{10}+\alpha^{8}+\alpha^{7}+\alpha^{5}+\alpha^{4}+\alpha^{3}+\alpha^{2})\\
&=g(\beta^{2}+\beta^{20}+\beta^{6}+\beta+\beta^{10}+\beta^{19}+\beta^{5})\\
&=g(\beta^{2})+g(\beta^{20})+g(\beta^{6})+g(\beta)+g(\beta^{10})+g(\beta^{19})+g(\beta^{5})\\
&=\beta^{2}+\beta^{22}+\beta^{6}+\beta+\beta^{7}+\beta^{5}+\beta^{4}\\
&=\alpha^{10}+\alpha^{18}+\alpha^{7}+\alpha^{5}+\alpha^{12}+\alpha^{2}+\alpha^{20}\\
&=\alpha^{10}+(\alpha^{10}+\alpha^{9}+\alpha^{6}+\alpha^{5}+\alpha^{3}+\alpha^{2})+\alpha^{7}+\alpha^{5}+\alpha^{12}+\alpha^{2}+\alpha^{20}\\
&=\alpha^{9}+\alpha^{7}+\alpha^{6}+\alpha^{3}+\alpha^{12}+\alpha^{20}\\
&=\alpha^{9}+\alpha^{7}+\alpha^{6}+\alpha^{3}+(\alpha^{10}+\alpha^{8}+\alpha^{7}+\alpha^{6}+\alpha^{2}+\alpha)+\alpha^{20}\\
&=\alpha^{10}+\alpha^{9}+\alpha^{8}+\alpha^{3}+\alpha^{2}+\alpha+\alpha^{20}\\
&=\alpha^{10}+\alpha^{9}+\alpha^{8}+\alpha^{3}+\alpha^{2}+\alpha+(\alpha^{10}+\alpha^{9}+\alpha^{7}+\alpha^{4}+\alpha^{2}+1)\\
&=\alpha^{8}+\alpha^{7}+\alpha^{4}+\alpha^{3}+\alpha+1\\
&=\alpha^{16},\qquad g(\beta^3)=\beta^{17},\qquad \text{YB, RL, NZ}
\end{aligned}$$

$$\begin{aligned}
g(\alpha^{16})&=g(\alpha^{8}+\alpha^{7}+\alpha^{4}+\alpha^{3}+\alpha+1)\\
&=g(\beta^{20}+\beta^{6}+\beta^{10}+\beta^{19}+\beta^{14}+\beta^{23})\\
&=g(\beta^{20})+g(\beta^{6})+g(\beta^{10})+g(\beta^{19})+g(\beta^{14})+g(\beta^{23})\\
&=\beta^{22}+\beta^{6}+\beta^{7}+\beta^{5}+\beta^{19}+\beta^{8}\\
&=\alpha^{18}+\alpha^{7}+\alpha^{12}+\alpha^{2}+\alpha^{3}+\alpha^{17}\\
&=(\alpha^{10}+\alpha^{9}+\alpha^{6}+\alpha^{5}+\alpha^{3}+\alpha^{2})+\alpha^{7}+\alpha^{12}+\alpha^{2}+\alpha^{3}+\alpha^{17}\\
&=\alpha^{10}+\alpha^{9}+\alpha^{7}+\alpha^{6}+\alpha^{5}+\alpha^{12}+\alpha^{17}\\
&=\alpha^{10}+\alpha^{9}+\alpha^{7}+\alpha^{6}+\alpha^{5}+(\alpha^{10}+\alpha^{8}+\alpha^{7}+\alpha^{6}+\alpha^{2}+\alpha)+\alpha^{17}\\
&=\alpha^{9}+\alpha^{8}+\alpha^{5}+\alpha^{2}+\alpha+\alpha^{17}\\
&=\alpha^{9}+\alpha^{8}+\alpha^{5}+\alpha^{2}+\alpha+(\alpha^{9}+\alpha^{8}+\alpha^{5}+\alpha^{4}+\alpha^{2}+\alpha)\\
&=\alpha^{4}\\
&=\alpha^4,\qquad g(\beta^{17})=\beta^{10},\qquad\text{YB, RL, NZ}
\end{aligned}$$

$$\begin{aligned}
g(\alpha^{17})&=g(\alpha^{9}+\alpha^{8}+\alpha^{5}+\alpha^{4}+\alpha^{2}+\alpha)\\
&=g(\beta^{11}+\beta^{20}+\beta+\beta^{10}+\beta^{5}+\beta^{14})\\
&=g(\beta^{11})+g(\beta^{20})+g(\beta)+g(\beta^{10})+g(\beta^{5})+g(\beta^{14})\\
&=\beta^{12}+\beta^{22}+\beta+\beta^{7}+\beta^{4}+\beta^{19}\\
&=\alpha^{14}+\alpha^{18}+\alpha^{5}+\alpha^{12}+\alpha^{20}+\alpha^{3}\\
&=(\alpha^{9}+\alpha^{7}+\alpha^{6}+\alpha^{4}+\alpha^{3}+\alpha^{2}+\alpha)+\alpha^{18}+\alpha^{5}+\alpha^{12}+\alpha^{20}+\alpha^{3}\\
&=\alpha^{9}+\alpha^{7}+\alpha^{6}+\alpha^{5}+\alpha^{4}+\alpha^{2}+\alpha+\alpha^{18}+\alpha^{12}+\alpha^{20}\\
&=\alpha^{9}+\alpha^{7}+\alpha^{6}+\alpha^{5}+\alpha^{4}+\alpha^{2}+\alpha+(\alpha^{10}+\alpha^{9}+\alpha^{6}+\alpha^{5}+\alpha^{3}+\alpha^{2})+\alpha^{12}+\alpha^{20}\\
&=\alpha^{10}+\alpha^{7}+\alpha^{4}+\alpha^{3}+\alpha+\alpha^{12}+\alpha^{20}\\
&=\alpha^{10}+\alpha^{7}+\alpha^{4}+\alpha^{3}+\alpha+(\alpha^{10}+\alpha^{8}+\alpha^{7}+\alpha^{6}+\alpha^{2}+\alpha)+\alpha^{20}\\
&=\alpha^{8}+\alpha^{6}+\alpha^{4}+\alpha^{3}+\alpha^{2}+\alpha^{20}\\
&=\alpha^{8}+\alpha^{6}+\alpha^{4}+\alpha^{3}+\alpha^{2}+(\alpha^{10}+\alpha^{9}+\alpha^{7}+\alpha^{4}+\alpha^{2}+1)\\
&=\alpha^{10}+\alpha^{9}+\alpha^{8}+\alpha^{7}+\alpha^{6}+\alpha^{3}+1\\
&=\alpha^{21},\qquad g(\beta^8)=\beta^{18},\qquad \text{YB, NZ, RL}
\end{aligned}$$ 

$$\begin{aligned}
g(\alpha^{18})&=g(\alpha^{10}+\alpha^{9}+\alpha^{6}+\alpha^{5}+\alpha^{3}+\alpha^{2})\\
&=g(\beta^{2}+\beta^{11}+\beta^{15}+\beta+\beta^{19}+\beta^{5})\\
&=g(\beta^{2})+g(\beta^{11})+g(\beta^{15})+g(\beta)+g(\beta^{19})+g(\beta^{5})\\
&=\beta^{2}+\beta^{12}+\beta^{20}+\beta+\beta^{5}+\beta^{4}\\
&=\alpha^{10}+\alpha^{14}+\alpha^{8}+\alpha^{5}+\alpha^{2}+\alpha^{20}\\
&=\alpha^{10}+(\alpha^{9}+\alpha^{7}+\alpha^{6}+\alpha^{4}+\alpha^{3}+\alpha^{2}+\alpha)+\alpha^{8}+\alpha^{5}+\alpha^{2}+\alpha^{20}\\
&=\alpha^{10}+\alpha^{9}+\alpha^{8}+\alpha^{7}+\alpha^{6}+\alpha^{5}+\alpha^{4}+\alpha^{3}+\alpha+\alpha^{20}\\
&=\alpha^{10}+\alpha^{9}+\alpha^{8}+\alpha^{7}+\alpha^{6}+\alpha^{5}+\alpha^{4}+\alpha^{3}+\alpha+(\alpha^{10}+\alpha^{9}+\alpha^{7}+\alpha^{4}+\alpha^{2}+1)\\
&=\alpha^{8}+\alpha^{6}+\alpha^{5}+\alpha^{3}+\alpha^{2}+\alpha+1\\
&=\alpha^{13},\qquad g(\beta^{22})=\beta^{21},\qquad \text{YB, NZ}
\end{aligned}$$ 

$$\begin{aligned}
g(\alpha^{19})&=g(\alpha^{10}+\alpha^{9}+\alpha^{5}+\alpha^{4}+\alpha^{3}+\alpha+1)\\
&=g(\beta^{2}+\beta^{11}+\beta+\beta^{10}+\beta^{19}+\beta^{14}+\beta^{23})\\
&=g(\beta^{2})+g(\beta^{11})+g(\beta)+g(\beta^{10})+g(\beta^{19})+g(\beta^{14})+g(\beta^{23})\\
&=\beta^{2}+\beta^{12}+\beta+\beta^{7}+\beta^{5}+\beta^{19}+\beta^{8}\\
&=\alpha^{10}+\alpha^{14}+\alpha^{5}+\alpha^{12}+\alpha^{2}+\alpha^{3}+\alpha^{17}\\
&=\alpha^{10}+(\alpha^{9}+\alpha^{7}+\alpha^{6}+\alpha^{4}+\alpha^{3}+\alpha^{2}+\alpha)+\alpha^{5}+\alpha^{12}+\alpha^{2}+\alpha^{3}+\alpha^{17}\\
&=\alpha^{10}+\alpha^{9}+\alpha^{7}+\alpha^{6}+\alpha^{5}+\alpha^{4}+\alpha+\alpha^{12}+\alpha^{17}\\
&=\alpha^{10}+\alpha^{9}+\alpha^{7}+\alpha^{6}+\alpha^{5}+\alpha^{4}+\alpha+(\alpha^{10}+\alpha^{8}+\alpha^{7}+\alpha^{6}+\alpha^{2}+\alpha)+\alpha^{17}\\
&=\alpha^{9}+\alpha^{8}+\alpha^{5}+\alpha^{4}+\alpha^{2}+\alpha^{17}\\
&=\alpha^{9}+\alpha^{8}+\alpha^{5}+\alpha^{4}+\alpha^{2}+(\alpha^{9}+\alpha^{8}+\alpha^{5}+\alpha^{4}+\alpha^{2}+\alpha)\\
&=\alpha,\qquad g(\beta^{13})=\beta^{14},\qquad\text{YB, NZ}
\end{aligned}$$ 

$$\begin{aligned}
g(\alpha^{20})&=g(\alpha^{10}+\alpha^{9}+\alpha^{7}+\alpha^{4}+\alpha^{2}+1)\\
&=g(\beta^{2}+\beta^{11}+\beta^{6}+\beta^{10}+\beta^{5}+\beta^{23})\\
&=g(\beta^{2})+g(\beta^{11})+g(\beta^{6})+g(\beta^{10})+g(\beta^{5})+g(\beta^{23})\\
&=\beta^{2}+\beta^{12}+\beta^{6}+\beta^{7}+\beta^{4}+\beta^{8}\\
&=\alpha^{10}+\alpha^{14}+\alpha^{7}+\alpha^{12}+\alpha^{20}+\alpha^{17}\\
&=\alpha^{10}+(\alpha^{9}+\alpha^{7}+\alpha^{6}+\alpha^{4}+\alpha^{3}+\alpha^{2}+\alpha)+\alpha^{7}+\alpha^{12}+\alpha^{20}+\alpha^{17}\\
&=\alpha^{10}+\alpha^{9}+\alpha^{6}+\alpha^{4}+\alpha^{3}+\alpha^{2}+\alpha+\alpha^{12}+\alpha^{20}+\alpha^{17}\\
&=\alpha^{10}+\alpha^{9}+\alpha^{6}+\alpha^{4}+\alpha^{3}+\alpha^{2}+\alpha+(\alpha^{10}+\alpha^{8}+\alpha^{7}+\alpha^{6}+\alpha^{2}+\alpha)+\alpha^{20}+\alpha^{17}\\
&=\alpha^{9}+\alpha^{8}+\alpha^{7}+\alpha^{4}+\alpha^{3}+\alpha^{20}+\alpha^{17}\\
&=\alpha^{9}+\alpha^{8}+\alpha^{7}+\alpha^{4}+\alpha^{3}+(\alpha^{10}+\alpha^{9}+\alpha^{7}+\alpha^{4}+\alpha^{2}+1)+\alpha^{17}\\
&=\alpha^{10}+\alpha^{8}+\alpha^{3}+\alpha^{2}+1+\alpha^{17}\\
&=\alpha^{10}+\alpha^{8}+\alpha^{3}+\alpha^{2}+1+(\alpha^{9}+\alpha^{8}+\alpha^{5}+\alpha^{4}+\alpha^{2}+\alpha)\\
&=\alpha^{10}+\alpha^{9}+\alpha^{5}+\alpha^{4}+\alpha^{3}+\alpha+1\\
&=\alpha^{19},\qquad g(\beta^{4})=\beta^{13},\qquad \text{YB, NZ}
\end{aligned}$$

$$\begin{aligned}
g(\alpha^{21})&=g(\alpha^{10}+\alpha^{9}+\alpha^{8}+\alpha^{7}+\alpha^{6}+\alpha^{3}+1)\\
&=g(\beta^{2}+\beta^{11}+\beta^{20}+\beta^{6}+\beta^{15}+\beta^{19}+\beta^{23})\\
&=g(\beta^{2})+g(\beta^{11})+g(\beta^{20})+g(\beta^{6})+g(\beta^{15})+g(\beta^{19})+g(\beta^{23})\\
&=\beta^{2}+\beta^{12}+\beta^{22}+\beta^{6}+\beta^{20}+\beta^{5}+\beta^{8}\\
&=\alpha^{10}+\alpha^{14}+\alpha^{18}+\alpha^{7}+\alpha^{8}+\alpha^{2}+\alpha^{17}\\
&=\alpha^{10}+(\alpha^{9}+\alpha^{7}+\alpha^{6}+\alpha^{4}+\alpha^{3}+\alpha^{2}+\alpha)+\alpha^{18}+\alpha^{7}+\alpha^{8}+\alpha^{2}+\alpha^{17}\\
&=\alpha^{10}+\alpha^{9}+\alpha^{8}+\alpha^{6}+\alpha^{4}+\alpha^{3}+\alpha+\alpha^{18}+\alpha^{17}\\
&=\alpha^{10}+\alpha^{9}+\alpha^{8}+\alpha^{6}+\alpha^{4}+\alpha^{3}+\alpha+(\alpha^{10}+\alpha^{9}+\alpha^{6}+\alpha^{5}+\alpha^{3}+\alpha^{2})+\alpha^{17}\\
&=\alpha^{8}+\alpha^{5}+\alpha^{4}+\alpha^{2}+\alpha+\alpha^{17}\\
&=\alpha^{8}+\alpha^{5}+\alpha^{4}+\alpha^{2}+\alpha+(\alpha^{9}+\alpha^{8}+\alpha^{5}+\alpha^{4}+\alpha^{2}+\alpha)\\
&=\alpha^{9}\\
&=\alpha^{9},\qquad g(\beta^{18})=\beta^{11},\qquad \text{YB, NZ}
\end{aligned}$$

$$\begin{aligned}
g(\alpha^{22})&=g(\alpha^{10}+\alpha^{8}+\alpha^{6}+\alpha^{5}+\alpha^{4}+1)\\
&=g(\beta^{2}+\beta^{20}+\beta^{15}+\beta+\beta^{10}+\beta^{23})\\
&=g(\beta^{2})+g(\beta^{20})+g(\beta^{15})+g(\beta)+g(\beta^{10})+g(\beta^{23})\\
&=\beta^{2}+\beta^{22}+\beta^{20}+\beta+\beta^{7}+\beta^{8}\\
&=\alpha^{10}+\alpha^{18}+\alpha^{8}+\alpha^{5}+\alpha^{12}+\alpha^{17}\\
&=\alpha^{10}+(\alpha^{10}+\alpha^{9}+\alpha^{6}+\alpha^{5}+\alpha^{3}+\alpha^{2})+\alpha^{8}+\alpha^{5}+\alpha^{12}+\alpha^{17}\\
&=\alpha^{9}+\alpha^{8}+\alpha^{6}+\alpha^{3}+\alpha^{2}+\alpha^{12}+\alpha^{17}\\
&=\alpha^{9}+\alpha^{8}+\alpha^{6}+\alpha^{3}+\alpha^{2}+(\alpha^{10}+\alpha^{8}+\alpha^{7}+\alpha^{6}+\alpha^{2}+\alpha)+\alpha^{17}\\
&=\alpha^{10}+\alpha^{9}+\alpha^{7}+\alpha^{3}+\alpha+\alpha^{17}\\
&=\alpha^{10}+\alpha^{9}+\alpha^{7}+\alpha^{3}+\alpha+(\alpha^{9}+\alpha^{8}+\alpha^{5}+\alpha^{4}+\alpha^{2}+\alpha)\\
&=\alpha^{10}+\alpha^{8}+\alpha^{7}+\alpha^{5}+\alpha^{4}+\alpha^{3}+\alpha^{2}\\
&=\alpha^{15},\qquad g(\beta^{9})=\beta^{3},\qquad \text{YB, NZ}
\end{aligned}$$

\end{proof}

\begin{rem}
The additivity of $g$ was first tested for $\beta:=\alpha$ and found to be not additive. We then tried $\beta:=\alpha^5$ because it was unrelated to $\alpha$ under the squaring function and suspect that those obtained from successive squaring of $\alpha^5$ should also produce an additive $g$ while those which aren't do not.
\end{rem}

\subsection{Matrices}
Since the functions $f$ and $g$ can be extended to additive functions on $\F_{2^{11}}$ we can express them as matrices with respect to the bases $A$ and $\chi$ ordered as they were written in their definitions.

\setcounter{MaxMatrixCols}{11}

$$[g] _{A}^{A} = \begin{pmatrix}
0&1&0&0&0&0&0&0&0&0&0 \\
0&0&0&1&0&0&0&0&0&1&0 \\
1&0&1&1&1&0&0&0&1&1&0 \\
0&1&0&0&0&0&0&0&1&1&0 \\
1&0&1&0&0&0&0&0&0&1&0\\
1&0&0&0&0&1&0&0&1&0&0\\
0&0&0&0&1&0&0&0&1&1&0\\
0&0&1&0&1&0&0&1&0&1&0\\
1&0&0&0&1&0&1&0&0&0&0\\
1&0&1&0&0&0&0&0&1&1&0\\
0&0&1&0&1&0&0&0&1&0&1\\
\end{pmatrix}$$

$$[f] _{A}^{A} = \begin{pmatrix}
0 & 0 & 0 & 0 & 0 & 0 & 1 & 0 & 1 & 0 & 0 \\
0 & 0 & 0 & 0 & 0 & 0 & 1 & 1 & 1 & 1 & 0 \\
0 & 0 & 0 & 0 & 0 & 0 & 0 & 1 & 1 & 1 & 1 \\
0 & 0 & 0 & 0 & 0 & 0 & 0 & 0 & 1 & 1 & 1 \\
0 & 0 & 0 & 0 & 0 & 0 & 0 & 0 & 0 & 1 & 1 \\ 
1 & 0 & 0 & 0 & 0 & 0 & 1 & 0 & 1 & 0 & 1 \\
0 & 1 & 0 & 0 & 0 & 0 & 1 & 1 & 1 & 1 & 0 \\
0 & 0 & 1 & 0 & 0 & 0 & 1 & 1 & 0 & 1 & 1 \\
0 & 0 & 0 & 1 & 0 & 0 & 0 & 1 & 1 & 0 & 1 \\
0 & 0 & 0 & 0 & 1 & 0 & 1 & 0 & 0 & 1 & 0 \\
0 & 0 & 0 & 0 & 0 & 1 & 0 & 1 & 0 & 0 & 1 \\
\end{pmatrix}$$

$$[f] _{\chi}^{\chi} = \begin{pmatrix}
1 & 0 & 0 & 0 & 0 & 1 & 1 & 0 & 0 & 0 & 0 \\
0 & 0 & 0 & 1 & 1 & 1 & 0 & 1 & 1 & 0 & 0 \\
0 & 0 & 1 & 0 & 0 & 1 & 0 & 0 & 0 & 1 & 0 \\
0 & 0 & 0 & 1 & 0 & 1 & 0 & 0 & 0 & 1 & 0 \\
0 & 0 & 0 & 1 & 0 & 1 & 1 & 0 & 0 & 1 & 1 \\
1 & 0 & 0 & 0 & 1 & 1 & 0 & 1 & 1 & 0 & 1 \\
1 & 1 & 0 & 0 & 0 & 0 & 0 & 1 & 1 & 0 & 0 \\
1 & 1 & 1 & 1 & 0 & 0 & 0 & 1 & 1 & 1 & 0 \\
1 & 1 & 1 & 0 & 1 & 0 & 0 & 0 & 0 & 1 & 1 \\
0 & 1 & 1 & 1 & 0 & 0 & 0 & 0 & 0 & 1 & 1 \\
0 & 0 & 1 & 1 & 1 & 1 & 0 & 0 & 1 & 0 & 1 \\
\end{pmatrix}$$

$$[g]_{\chi}^{\chi} = \begin{pmatrix}
0 & 0 & 0 & 1 & 1 & 1 & 1 & 0 & 0 & 0 & 0 \\
0 & 1 & 0 & 0 & 0 & 1 & 0 & 1 & 1 & 0 & 0 \\
0 & 0 & 0 & 0 & 1 & 1 & 0 & 0 & 0 & 1 & 0 \\
0 & 1 & 0 & 1 & 0 & 1 & 0 & 0 & 0 & 1 & 0 \\
0 & 0 & 1 & 1 & 1 & 1 & 1 & 0 & 0 & 1 & 1 \\
0 & 0 & 0 & 0 & 1 & 1 & 0 & 1 & 1 & 0 & 1 \\
1 & 0 & 0 & 1 & 1 & 0 & 0 & 1 & 1 & 0 & 0 \\
1 & 0 & 0 & 1 & 0 & 0 & 0 & 1 & 1 & 1 & 0 \\
0 & 0 & 1 & 0 & 0 & 0 & 0 & 0 & 0 & 1 & 1 \\
1 & 1 & 1 & 0 & 1 & 0 & 0 & 0 & 0 & 1 & 1 \\
1 & 0 & 1 & 1 & 0 & 1 & 0 & 0 & 1 & 0 & 1 \\
\end{pmatrix}$$


\begin{thm}There is a faithful 11 dimensional irreducible representation of $M_{23}$ over $\F_2$ such that the generators $f$ and $g$ are given by the matrices $[f]_{\chi}^{\chi}$ and $[g]_{\chi}^{\chi}$ with respect to $\chi$, while it is given by $[f] _{A}^{A}$ and $[g] _{A}^{A}$ with respect to $A$. The $23$ element subgroup of the multiplicative group of $\F_{2^{11}}$ is preserved by $M_{23}$ under this representation and the restrictions of $f$ to $C$ and $g$ to $C$ under the identification
$$j\mapsto \alpha^{5j}$$
is given by
$$f|_C=(1,2,3,\dots,22,23)$$
and 
$$g|_C=(3,17,10,7,9)(4,13,14,19,5)(8,18,11,12,23)(15,20,22,21,16).$$
Furthermore, 11 is the minimal dimension of a faithful $\F_2$ representation.

\end{thm}

\begin{proof} Since $f$ and $g$ correspond to generators of $M_{23}$ as permutations in $S_{23}$, they will also generate $M_{23}$ as functions on $\F_{2^{11}}$. (Faithful).
Moreover, since $f$ and $g$ are additive ($\F_2$ linear), the functions they generate will also be additive and hence we have a representation of $M_{23}$.

Now since this representation is faithful, it suffices to show that there are no smaller faithful representations and this will imply irreducibility as well.
To do this it suffices to observe that for the size of $M_{23}$ to divide the size of $GL$, we need $23| (2^i-1)$ for some $i$ but this forces $i\geq11$.

\end{proof}

\begin{rem}
Since $M_{23}$ is simple, every non-trivial representation is faithful.
\end{rem}

\begin{rem}
We also tested to see if this representation extends to $M_{24}$ using the generators $f$, $g$ and
$$h= (1,24)(2,23)(3,12)(4,16)(5,18)(6,10)(7,20)(8,14)(9,21)(11,17)(13,22)(15,19),$$ 
which appeared in \cite{M1873} with $B=21A$, $C=A$ and $D=-A$, but $h$ does not seem to be additive for any choice of $\beta$ obtained from $\alpha^5$ by successive squaring. In each case we tested an additivity relation that does not involve $\beta$ and obtained a contradiction of the form $\beta=0$. 
\end{rem}


\section*{Acknowledgements} We are very grateful for the encouragement and support of the parents of our talented young researchers many of whom are not yet even in middle school. Honourable mention to Gavin Tran, Rhianna Kho and Max Lau for participating.

We also acknowledge free resources such as Wikipedia, ATLAS, groupprops, MathStackExchange, MathOverflow and open access journals for their assistance. This project was completed without access to journal subscriptions.

\printbibliography

@article{M1873,
    author= {Emile Mathieu},
    title = {Sur la fonction cinq fois transitive de 24 quantités.[On the Five times transitive function of 24 quantities]},
    journal = {Journal de mathématiques pures et appliquées 2e série},
    volume ={18},
    pages ={25-46},
    year ={1873} %RH,YB
}

@article{trans,
   author= {Y. Bing and B. Hu and R. Hu and R. Kho and R. Li and S. Lu and F. Mcdonald and M. Sun and N. Wolfe and J. Yao and L. Zhou and N. Zhou},
   title = {French to English Mathieu 1873 Translation},
   journal = {placeholder journal},
   volume = {placeholder volume},
   pages = {placeholder pages},
   year = {0000},
   URL = {https://placeholder}
}

@article{P1957,
   author= {L. J. Paige},
   title = {A Note on Mathieu Groups},
   journal = {Canadian Journal of Mathematics},
   volume = {9},
   pages = {15-18},
   year = {1957}
}

@article{C1935,
    author= {R. Church},
    title = {Tables of Irreducible Polynomials for the First Four Prime Moduli},
    journal = {Annals of Mathematics},
    volume = {36},
    pages = {198-209},
    year = {1935}
}

{\sc Dr Michael Sun's School of Maths, NSW, Australia}

Corresponding author:
\email{humanyimingbing@gmail.com}

%
%
\end{document}